\documentclass{amsart}
\usepackage{amssymb}
\usepackage{xy}

\xyoption{all}

\newif\iffurther
\furtherfalse

\newtheorem{thm}{Theorem}[section] 

\newtheorem{cor}[thm]{Corollary}
\newtheorem{defn}[thm]{Definition}

\newtheorem{exmpl}[thm]{Example}
\newtheorem{lem}[thm]{Lemma}
\newtheorem{prop}[thm]{Proposition}

\newtheorem{rem}[thm]{Remark}

\newtheorem{ques}[thm]{Question}

\newcommand\software[1]{{\texttt{#1}}} 

\def\[{\left[}
\def\]{\right]}

\def\almostprime{union-prime}
\def\GK{{\operatorname{GKdim}}}
\def\GKdim{{Gel'fand-Kirillov dimension}}
\newcommand\sg[1]{{\left<#1\right>}}
\newcommand\ideal[1]{{\left<#1\right>}}
\long\def\forget#1\forgotten{{}}

\newcommand\suchthat{{\,:\ }}
\newcommand\M[1][n]{{\operatorname{M}_{#1}}}
\def\sub{\subseteq}
\def\nsub{\,\not\subseteq\,}
\newcommand\isom{{\,\cong\,}}
\def\ra{{\rightarrow}}

\newcommand\set[1]{{\{#1\}}}
\def\N{{\mathbb{N}}}

\def\co{{\,:\,}}
\newcommand\smat[4]{{\left(\begin{array}{cc} {#1} & {#2} \\ {#3} & {#4} \end{array}\right)}}
\newcommand\End{{\operatorname{End}}}
\newcommand\Rref[1]{{Remark~\ref{#1}}}
\newcommand\Eref[1]{{Example~\ref{#1}}}
\newcommand\Pref[1]{{Proposition~\ref{#1}}}
\newcommand\Cref[1]{{Corollary~\ref{#1}}}

\newcommand\Lref[1]{{Lemma~\ref{#1}}}
\newcommand\Sref[1]{{Section~\ref{#1}}}
\newcommand\eq[1]{{(\ref{#1})}}

\def\normali{{\lhd}} 
\def\normal{{\unlhd}} 
\def\lam{{\lambda}}
\def\Z{{\mathbb{Z}}}
\def\COND{{(\footnotesize{\diamondsuit})}}
\def\PIdeg{{\operatorname{PI}}}
\def\PPind{{\operatorname{\PP}}}

\newcommand\mul[1]{{{#1}^{\times}}}

\newcommand\tensor[1][]{{\otimes_{#1}}}
\renewcommand\span{\operatorname{span}}
\DeclareMathOperator{\spec}{spec}
\def\({\left(}
\def\){\right)}
\title{Unions of chains of primes
\iffurther{\\ -- NOT FOR DISTRIBUTION --}\fi
}

\author{Be'eri Greenfeld}
\author{Louis H.~Rowen}
\author{Uzi Vishne}
\address{Department of Mathematics, Bar Ilan University, Ramat Gan 5290002, Israel}
\email{beeri.greenfeld@gmail.com, \{rowen,vishne\}@math.biu.ac.il}

\thanks{This research was partially supported by a BSF grant no. 206940}

\thanks {We thank L. Small for many helpful suggestions, including
pointing out \cite[Exmpl.~4.2]{P}.}

\date{\today}

\begin{document}

\begin{abstract}
The union of an ascending chain of prime ideals is not always prime.
The union of an ascending chain of semi-prime ideals is not always semi-prime.
We show that these two properties are independent.
We also show that the number of non-prime unions of subchains in a chain of primes in a PI-algebra
does not exceed the PI-class minus one, and this bound is tight.
\end{abstract}

\maketitle

\newcommand\union[2]{{\mbox{$\bigcup{#1}={#2}$}}}


\def\parl{{}}
\def\parr{{}}
\def\PP{{\mbox{\parl$\mathcal P^{\uparrow}$\parr}}}
\def\PSP{{\mbox{\parl$\mathcal P$\parr}}}
\def\SPSP{{\mbox{\parl$\mathcal S\!P^{\uparrow}$\parr}}}
\def\Zent{{\operatorname{Cent}}}

\section{Introduction}\label{sec:intro}

In a commutative ring, the union of a chain of prime ideals is
prime, and the union of a chain of semiprime ideals is semiprime.
This paper demonstrates and measures the failure of these chain
conditions in general.


\begin{defn} A ring has the {\bf{(semi)prime chain property}} (denoted
\PP\ and \SPSP, respectively) if the union of any countable chain
of (semi)prime ideals is always (semi)prime.\footnote{For
simplicity we deal only with countable chains throughout the
paper, but the arguments are general.}
\end{defn}

The property \SPSP\ was recognized by Fisher and Snider \cite{FS} as the
missing hypothesis for Kaplansky's conjecture on regular rings,
and they gave an example of a ring without \SPSP.

Our focus is on \PP. The class of rings satisfying \PP\ is quite large. An easy exercise shows that every commutative ring satisfies \PP, and the same
argument yields that the union of strongly prime ideals is strongly prime ($P \normali R$ is strongly prime if $R/P$ is a domain).
In fact, we have the following result:

\begin{prop}\label{first}
Every ring $R$ which is a finite module over a central subring, satisfies \PP.
\end{prop}
\begin{proof}
Write $R = \sum_{i=1}^t Cr_i$ where $C \sub
\operatorname{Cent}(R)$. Suppose $P_1 \subset P_2 \subset \cdots$ is
a chain of prime ideals, with $P = \cup P_i$. If $a,b \in R$ with
$$\sum C ar_i b = \sum aCr_i b = aRb \subseteq P,$$ then there is
$n$ such that $ar_i b \in P_n$ for $1 \le i \le t,$ implying $aRb =
\sum Car_i b \subseteq P_n$, and thus $a \in P_n$ or $b \in P_n$.
\end{proof}

(For a recent treatment of the correspondence of infinite chains of primes between a ring $R$ and a central subring, see \cite{Shai}).

The class of rings satisfying \PP\ also contains every ring that satisfies ACC (ascending chain condition) on primes, and is closed under homomorphic images and
central localizations.
This led some mathematicians to believe that it holds in general.
%
On the other hand, Bergman produced an example lacking \PP\ (see
\Eref{Ex1} below), implying that the free algebra on two generators does not have~\PP.

Obviously, the property \PP\ follows from the maximum property on families of primes. On the other hand, \PP\ implies (by Zorn's lemma) the following maximum property: for every prime $Q$ contained in any ideal $I$, there is a prime $P$ maximal with respect to $Q \sub P \sub I$.

In \Sref{sec:mat} we show that \PP\ and \SPSP\ are independent, by presenting an example (due to Kaplansky and Lanski) of a ring satisfying \PP\ and not \SPSP, and an example of a ring satisfying \SPSP\ but not \PP.

We say that an ideal is {\bf{\almostprime}} if it is a union of a
chain of primes, but is not itself prime. (If $\set{P_\lam}$ is an
ascending chain of primes, then $R/\bigcup P_{\lambda} =
\lim_{\rightarrow} R/P_{\lam}$ is a direct limit of prime rings).
The {\bf{$\PP$-index}} of the ring $R$ is the maximal number of
non-prime unions of subchains of a chain of prime ideals in $R$ (or
infinity if the number is unbounded, see \Pref{PPindex}).
\Sref{sec:mon} extends Bergman's example by showing that the
\PP-index of the free (countable) algebra is infinity. A variation of this construction, based on free products, is presented in \Sref{sec:example2.2}. After defining the \PP-index in \Sref{sec:PP}, in
\Sref{sec:PI} we discuss PI-rings, showing that the \PP-index does
not exceed the PI-class minus one, and this bound is tight. We thank the anonymous referee for careful comments on a previous version of this paper.

\section{Monomial algebras}\label{sec:mon}

Fix a field $F$. We show that \PP\ and \SPSP\ fail in the free algebra (over $F$) by
constructing an (ascending) chain of primitive ideals whose union
is not semiprime. Let us start with a simpler theme, whose
variations have extra properties.

\begin{exmpl}[A chain of prime ideals with non-semiprime union]
\label{Ex1} Let $R$ be the free algebra in the (noncommuting)
variables $x,y$. For each $n$, let
$$P_n = \ideal{xx,xyx, xy^2x, \dots, xy^{n-1}x}.$$ 
As a monomial ideal, it is enough to check primality on monomials.
If $uRu' \sub P_n$ for some words $u,u'$, then in particular
$uy^{n}u' \in P_n$, which forces a subword of the form $xy^ix$
(with $i<n$) in $u$ or in $u'$; hence either $u\in P_n$ or $u'\in
P_n$.

On the other hand $\bigcup P_n = (RxR)^2$ which is not semiprime.
\end{exmpl}

This example, due to G.~Bergman, appears in \cite[Exmpl.~4.2]{P}.
Interestingly, primeness is always maintained in the following
sense (\cite[Lem.~4.1]{P}, also due to Bergman): for every
countable chain of primes $P_1 \subset P_2 \subset \cdots$ in a
ring $R$, the union $\bigcup (P_n[[\zeta]])$
is a prime
ideal of the power series ring $R[[\zeta]]$.

Since in
\Eref{Ex1} $\bigcup P_n = (RxR)^2$, if $Q \normali R$ is a prime
containing the union then $x \in Q$ so $R/Q$ is commutative.
In particular, a chain of prime ideals starting from the chain $P_1
\subset P_2 \subset \cdots$ has only one \almostprime. Let us
exhibit a (countable) chain providing infinitely many
\almostprime{}s.

\begin{exmpl}[A prime chain with infinitely many \almostprime{}s]
Let $R$ be the free algebra generated by $x,y,z$. For a monomial $w$ we denote by $\deg_yw$ the degree of $w$ with respect to $y$.
For $i,n \geq 1$, consider the monomial ideals
$$I_{i,n} = RxxR+RxzxR+\cdots+Rxz^{i-1}xR+\sum_{\deg_y w < n} R xz^ixwxz^ix R,$$
which form an ascending chain with respect to the lexicographic
order on the indices $(i,n)$, since $xz^ix \in I_{i',n}$ for every
$i'>i$. To show that $I_{i,n}$ are prime, suppose that $u,u'$ are
monomials such that $u,u' \not \in I_{i,n}$ but $u R u' \sub
I_{i,n}$. Then $u z^i y^n z^i u' \in I_{i,n}$. Since none of the
monomials $xz^{i'}x$ ($i'<i$) is a subword of $u$ or $u'$, they are
not subwords of $u z^i y^n z^i u'$, forcing $u z^i y^n z^i u'$ to
have a subword of the form $xz^ixwxz^ix$ where $\deg_yw < n$. It
follows that $z^iy^nz^i$ is a subword of $z^ixwxz^i$, contrary to
the degree assumption. Now, for every $i$, $$\bigcup_{n} I_{i,n} =
RxxR+RxzxR+\cdots+Rxz^{i-1}xR+(Rxz^ixR)^2,$$ which contains
$(Rxz^ixR)^2$ but not $Rxz^ixR$, so it is not semiprime.
\end{exmpl}

In particular the \PP-index of $R$ (see \Pref{PPindex}) is infinity.
In Section~\ref{sec:PI} we show that this phenomenon is impossible
in PI algebras: there, the number of \almostprime{}s in a prime
chain is bounded by the PI-class.

\begin{rem}
The ideals $P_n$ in \Eref{Ex1} are in fact primitive. Indeed, Bell and Colak \cite{Bell} proved that any finitely presented prime monomial algebra is either primitive or PI (also see \cite{Ok}), and $R/P_n$ contains a free subalgebra, e.g. $k\ideal{xy^n,xy^{n+1},\dots}$.
\end{rem}
The same effect can be achieved by using idempotents.

\forget
Let $R$ be a binomial algebra, namely a quotient of a free algebra with respect to relations of the form $w_1 = w_2$ (or $w_1 = 0$) where $w_1,w_2$ are monomials. We say that an ideal $I$ in $R$ is monomially prime if $uRu' \sub I$ forces $u \in I$ or $u' \in I$ for every two monomials $u,u'$.
\begin{lem}
Let $R$ be the free algebra generated by $e,y$ subject to the relation $e^2 = e$. Then every monomial ideal which is monomially prime is prime.
\end{lem}
\begin{proof}
Let $I$ be a monomial ideal which is monomially prime. Let $f,g \in R$ be any two elements such that $fRg \sub I$.
\end{proof}
\forgotten






\begin{exmpl}[A chain of primitive ideals with non-semiprime union] \label{Ex3}
Let $R$ be the free algebra in the variables $e,y$, modulo the
relation $e^2 = e$. Every monomial has a unique shortest
presentation as a word (replacing $e^2$ by $e$ throughout). Ordering
monomials first by length and then lexicographically, every element
$f$ has an upper monomial $\bar{f}$. Notice that $\overline{f y^n g}
= \bar{f}y^n \bar{g}$.

For each $n$, let
$$P_n = \ideal{eye, ey^2e, \dots, ey^{n-1}e}.$$
To show that  $P_n$ is a prime ideal, assume that $f y^n g \in
P_n$. Then $\bar{f}y^n \bar{g} = \overline{f y^n g} \in P_n$,
forcing $\bar{f} \in P_n$ or $\bar{g} \in P_n$ as in \Eref{Ex1}.
The claim follows by induction on the number of monomials.

To show that the ideal $P_n$ is primitive, it is enough by
\cite{LRS} to prove that $e(R/P_n)e$ is a primitive ring.
We construct an isomorphism between $e(R/P_n)e$ and the countably generated free algebra $F\ideal{z_0,z_1,\dots}$ by sending $ey^me$ for $m \geq n$ (which clearly generate a free algebra) to $z_{m-n}$. But the free algebra is primitive (see 
\cite[Prop.~11.23]{Lam}
).

On the other hand $\bigcup P_n = ReyReR = ReRyeR$, which contains $(ReyR)^2$ but not $ey$, so is not semiprime.
\end{exmpl}

\begin{rem}
We say that a ring is {\bf{uniquely-\PP}} if it has a unique minimal
prime over every chain of prime ideals. Since the intersection of a
descending chain of primes is prime, Zorn's lemma shows that there
are minimal primes over every ideal, in particular over every
\almostprime.

In the topology of the spectrum, a net $\set{P_{\lambda}}_{\lambda
\in \Lambda}$ of primes converges to a prime $Q$ if and only if
$\bigcap_{\lambda \in \Lambda}\bigcup_{\lambda' \geq \lambda}
P_{\lambda'} \subseteq Q$; in particular when $\set{P_{\lambda}}$ is
an ascending chain, $\lim P_\lambda = Q$ if and only if $\bigcup
P_\lambda \sub Q$. Therefore, the spectrum can identify minimal
primes over \almostprime{}s. It seems that the spectrum cannot
distinguish \PP\ from uniquely-\PP.

In the examples of this section, there is a unique minimal prime
over every \almostprime. In \Eref{Ex4} below the situation is
different: the \almostprime\ ideal constructed there is the
intersection of two primes containing it.
\end{rem}

\section{Prime ideals in free products}\label{sec:example2.2} 


In \Eref{Ex1} there are infinitely many (incomparable) prime ideals
lying over the chain. We modify this example, in order to obtain a
chain over which there is unique prime. In \Eref{Ex1} we considered
ideals of the free algebra, which can be written as a free product
$F[x] *_F F[y]$. The quotient over the radical of the union over the
chain is the ``second'' component $F[y]$, which we would like to
replace by the field $F(y)$. The proof that the ideals are prime is
somewhat delicate; we thank the referee for pointing this out.


\forget 
\begin{exmpl}[A chain of prime ideals whose union is not semiprime, although its radical
is a maximal ideal]\label{Ex1div}
Let $D$ be the quotient division ring of the free algebra $F\ideal{x,y}$.
Let $R$ be the subalgebra generated by $x$ and the subfield $F(y)$. Extend $\deg_y \co F[y] \ra \N$ to $\deg_y \co \mul{F(y)} \ra \Z$ in the obvious manner.
Similarly to the previous example, take the prime ideals
$$P_n = \ideal{xax \suchthat a \in F(y), \deg_y(a) < n}.$$ 
Again $P = \bigcup P_n = (RxR)^2$, which is not semiprime. But now $R/\sqrt P = R/\ideal{x} \cong F(y)$.
\end{exmpl}
\forgotten

Let $F$ be a field and let $A,B$ be $F$-algebras, with given vector
space decompositions $A = F \oplus A_0$ and $B = F \oplus B_0$. The
free product $A *_F B$ can be viewed as the tensor algebra $T(A_0
\oplus B_0) = F \oplus \bigoplus_{n\geq 1} (A_0\oplus B_0)^{\tensor n}$, modulo the relations $a \tensor a' = aa'$ and $b
\tensor b' = bb'$ for every $a,a' \in A$ and $b,b' \in B$.
We will omit the tensor symbol.

Fixing the decomposition $F[x] = F \oplus xF[x]$ and an arbitrary decomposition $B = F \oplus B_0$, we consider ideals of the free product $R = F[x] *_F B$. The tensor algebra is graded by $x$, once we declare that $\deg(b) = 0$ for every $b \in B_0$, and this grading induces a grading on $R$.

Let $W \sub B$ be a vector space containing $F$.
We say that $W$ is {\bf{restricted}} if for every finite dimensional subspace $V \sub B$ there is an element $b \in B$ such that $Vb \sub B_0$ and $Vb \not \sub W$; and an element $b' \in B$ such that $b'V \sub B_0$ and $b'V \not \sub W$.


\begin{thm}\label{main}
The ideal $P = RxWxR$ of $R$ is prime whenever $W \sub B$ is a restricted subspace.
\end{thm}
\begin{proof}
Write $W = F \oplus W_0$ where $W_0 = W \cap B_0$. 
Let $L'$ be the ideal of $R$ generated by $x^2$. For $n \geq 0$ let us denote the vector spaces
$$L_n = B x B_0 x B_0 \cdots x B_0 x B,$$
where the degree with
respect to $x$ is $n$; so that $L_0 = B$ and $L_1 = B x B$. Setting $L = \sum_{n\geq 0} L_n$, we have
that $R = L' \oplus L$.

Let $P_n = L_n \cap P$; so $P_0 = P_1 = 0$, and for $n \geq 2$,
$$P_n = \sum B x B_0 x \cdots x B_0 x W_0 x B_0 x \cdots x B_0 x B$$
where in each summand one of the intermediate entries is $W_0$ and all the others are equal to $B_0$. For example $P_2 = B x W_0 x B$ and $P_3 = B x W_0 x B_0 x B + B x B_0 x W_0 x B$. Now,
since $F \sub W$, we have that $L' = RxxR = RxFxR \sub RxWxR = P$, and we can compute:
\begin{eqnarray*}
P & = & L' + P \\
& = & L' + (L'+L)x W x(L'+L) \\
& = & L' + Lx W x L \\
& = & L' + Lx W_0 x L \\
& = & L ' + \sum_{d,d' \geq 0} L_d x W_0 x L_{d'} \\
& = & L ' + \sum_{n \geq 0} \(\sum_{d+d' = n} L_d x W_0 x L_{d'}\) \\
& = & L' + \sum_{n \geq 2} P_n,
\end{eqnarray*}
since modulo $L'$, $xBx \equiv xB_0x$ and $xWx \equiv xW_0 x$.

Let $m \geq 1$. As a vector space, $L_m \isom B \tensor B_0 \tensor \cdots \tensor B_0 \tensor B$ with $m+1$ factors. This isomorphism carries $P_m$ to $\sum B \tensor B_0 \tensor \cdots \tensor B_0 \tensor W_0 \tensor B_0 \tensor \cdots \tensor B_0 \tensor B$ as above, and there is an isomorphism
$$\psi_m \co L_m / P_m \,\longrightarrow\, B \tensor \overline{B_0} \tensor \cdots \tensor \overline{B_0} \tensor B$$
with $m-1$ factors of the form $\overline{B_0} = B_0/W_0$. The image of $g \in L_m$ in $\overline{L_m} = L_m/P_m$ will be denoted by $\overline{g}$, hoping that no confusion is incurred by the double usage of the over-line.

We need to show that $P$ is prime. Since $L' \sub P$, it suffices to show that if $f,f' \in L$
and $f,f' \not \in P$, then $fBf' \not \sub P$. Furthermore since
$R$ is graded with respect to $x$, and $P$ is a homogeneous ideal with respect to this grading, 
we may assume that $f,f'$ are homogeneous with respect to $x$, so we can write
$$f = \sum_i a_{0,i} x a_{1,i} x \cdots x a_{n,i} \in L_n$$
and $$f' = \sum_j a'_{0,j} x a'_{1,j} x \cdots x a'_{n',j} \in
L_{n'}$$ where $a_{0,i},a_{n,i}, a'_{0,j}, a'_{n',j} \in B$ and
$a_{t,i}, a'_{t',j} \in B_0$ for $0 < t < n$ and $0 < t' < n'$. Let
$V$ be the vector space spanned by all the $a_{n,i}$ and $V'$ the
vector space spanned by all the $a'_{0,j}$. We say that $f$ ``ends in $V$'' and $f'$ ``begins in $V'$''. By assumption, there are elements $b,b' \in B$ such that $Vb, b'V' \sub B_0$, while $Vb \not \sub W_0$ and $b'V' \not \sub W_0$.

Since $Vb, b'V' \sub B_0$, we have that $f b x b' f' \in L_{n+n'+1}$.
Consider the commutative diagram
$$\xymatrix@C=32pt{L_n \tensor L_{n'} \ar@{->}[d]^{m} \ar@{->}[r]^{\theta} &
     \overline{L_n} \tensor \overline{L_{n'}} \ar@{->}[r]^(0.24){\psi_n \tensor \psi_{n'}} \ar@{->}[d]^{\bar{m}} &
     (B \tensor \overline{B_0} \tensor \cdots \tensor  \overline{B_0}
     \tensor B) \tensor  (B \tensor \overline{B_0} \tensor \cdots \tensor \overline{B_0}
     \tensor B) \ar@{->}[d]
    \\
L_{n+n'+1} \ar@{->}[r] &  \overline{L_{n+n'+1}}  \ar@{->}[r]^(0.4){\psi_{n+n'+1}} & B \tensor\overline{B_0}  \tensor \cdots \tensor  \overline{B_0}
     \tensor B
     }$$
\forget 
$$\xymatrix{L_n \tensor L_{n'} \ar@{->}[r]^{m} \ar@{->}[d]^{\theta} & L_{n+n'+1} \ar@{->}[d] \\
     \overline{L_n} \tensor \overline{L_{n'}} \ar@{->}[d]^{\psi_n \tensor \psi_{n'}} \ar@{->}[r]^{\bar{m}} & \overline{L_{n+n'+1}}  \ar@{->}[d]^{\psi_{n+n'+1}}
     \\
     (B \tensor \cdots 
     \tensor B) \tensor  (B \tensor \cdots 
     \tensor B) \ar@{->}[r] & B \tensor \cdots 
     \tensor B
     }$$
\forgotten 
where the domain of definition of the top-to-bottom maps is the 
elements in $L_n \tensor L_{n'}$ such that the left
factor ends in $B_0$ and the right factor begins in $B_0$ (and not
merely in $B$), and their image. Here, $m(g \tensor g') = g x g'$ and 
$\bar{m}(\overline{g} \tensor \overline{g'}) = \overline{g x g'}$
which is easily checked to be well-defined. The right-most arrow is
reduction of the two intermediate factors along $B_0 \ra B_0/W_0$.

\forget
$$\xymatrix{L_n \tensor L_{n'} \ar@{->}[r]^(0.4){m} \ar@{->}[d]^{\theta} & (L'+L_{n+n'+1})/L' \ar@{->}[d]\ar@{->}[r]^(0.61)\cong & L_{n+n'+1} \ar@{->}[d] \\
     \overline{L_n} \tensor \overline{L_{n'}} \ar@{->}[d]^{\psi_n \tensor \psi_{n'}} \ar@{->}[r]^(0.3){\bar{m}} & (L'+L_{n+n'+1})/(L'+P_{n+n'+1}) \ar@{->}[r]^(0.7)\cong & \overline{L_{n+n'+1}}  \ar@{->}[d]^{\psi_{n+n'+1}}
     \\
     (B \tensor \cdots 
     \tensor B) \tensor  (B \tensor \cdots 
     \tensor B) \ar@{-->}[rr] & & B \tensor \cdots 
     \tensor B
     }$$
where: $m(g \tensor g') = g x g' + L'$; $\bar{m}(\overline{g}
\tensor \overline{g'}) = g x g' + P_{n+n'+1}+L'$ which is easily
checked to be well-defined. The isomorphisms to the right are the
natural ones, induced by the fact that $L' \cap L_{n+n'+1} = 0$;
the bottom arrow is reduction of the two intermediate factors along
$B_0 \ra B_0/W_0$, and is defined on the subspace of elements whose
two intermediate factors are indeed in $B_0$. \forgotten

Now consider the element $f b\tensor b'f' \in L_n \tensor L_{n'}$, for the given $f \in L_n$ and $f' \in L_{n'}$. By assumption $\theta (f b\tensor b'f') = \overline{fb} \tensor \overline{b'f'}$ is non-zero, because $\overline{fb}, \overline{bf'} \neq 0$. Furthermore $\psi_n(\overline{fb})$ ends in $B_0$ and $\psi_{n'}(\overline{b'f'})$ begins in $B_0$, so $\psi_n(\overline{fb}) \tensor \psi_{n'}(\overline{b'f'})$ is in the domain of definition of the right-most arrow, which takes this element to $\psi_{n+n'+1}(\overline{fbxbf'})$. It remains to show that this element is nonzero.
But $\psi_n(\overline{fb})$ does not end in $W_0$, and $\psi_{n'}(\overline{b'f'})$ does not begin in $W_0$; hence their images in $B \tensor \overline{B_0} \tensor \cdots \tensor \overline{B_0} \tensor B$ are nonzero, and their tensor product, equal to $\psi_{n+n'+1}(\overline{fbxbf'})$, is nonzero as well.
\end{proof}

\begin{prop}\label{3.2}
Let $(K,\nu)$ be a valued field containing $F$ as a field of
scalars. For any $m \geq 0$, $W = \set{k \in K \suchthat \nu(k) \geq
-m}$ is a restricted subspace of $K$ (where the decomposition $K = F \oplus K_0$ is arbitrary).
\end{prop}
\begin{proof}
We first claim that if $U \sub K$ is an $F$-vector subspace of finite codimension,
then $\set{\nu(u) \suchthat u \in U}$ is unbounded from below.
Indeed, choose any finite dimensional complement $U'$, and notice that $\set{\nu(u')\suchthat u' \in U'}$ is bounded from below; so if $\nu(u)$ were bounded for $u \in U$, then $\nu(k)$ would be bounded over the set of $k \in K$.

Let $V\neq 0$ be a finite dimensional space. The space of elements $y$ such that $Vy \sub K_0$ has finite codimension, so by the previous argument contains elements of arbitrarily small value, for which $Vy \not \sub W$.
\end{proof}

\begin{cor}\label{almostexample}
Let $(K,\nu)$ be a valued field containing $F$ as a field of scalars. For fixed $m \geq 0$, let $W_m = \set{k \in K \suchthat \nu(k) \geq -m}$. Then the ideal generated by $xW_mx$ in $R = F[x] *_F K$ is prime.
\end{cor}

With the notation of \Cref{almostexample}, we now formulate the promised counterexample:
\begin{exmpl}[radical of a chain union which is a maximal ideal]\label{6.4}
Let $F$, $K$, $R$ and the $W_m$ be as above. By definition $\bigcup_{m\geq 0} W_m = K$.
Let $P_m$ be the (prime) ideal generated by $xW_mx$. Then $P_1 \sub P_2 \sub \cdots$ is a chain of prime ideals in $R$, and
$\bigcup_{m\geq 0} RxW_mxR = RxKxR = (RxR)^2$. The radical $RxR$ is thus maximal, as $R/RxR \,\isom\, K$.
\end{exmpl}

\section{Matrix constructions}\label{sec:mat}

This section shows that \PP\ and \SPSP\ are independent: the algebra in \Eref{ExKL} satisfies \PP\ but not \SPSP, and the algebra in \Eref{Ex4} satisfies \SPSP\ but not \PP.

\subsection{\PP\ does not imply \SPSP}

As mentioned in the introduction, Kaplansky conjectured that a
semiprime ring all of whose prime quotients are von Neumann regular,
is regular. Fisher and Snider~ \cite{FS} proved that this is the
case if the ring satisfies \SPSP\ (also see \cite[Thm.~1.17]{G}),
and gave a counterexample which lacks this property, due to
Kaplansky and Lanski \cite[Example~1.19]{G}. We repeat the example
and exhibit, in this ring, an explicit ascending chain of semiprime
ideals whose union is not semiprime.


\begin{exmpl}[Kaplansky-Lanski]\label{ExKL}
(A ring whose prime ideals are maximal, but without \SPSP). Let
$R$ be the ring of sequences of $2$-by-$2$ matrices which
eventually have the form $\begin{pmatrix}
\alpha & \beta_n \\
0 & \alpha
\end{pmatrix}$ in the $n$th place, clearly a semiprime ring.

Let $I_n$ be the set of sequences in $R$, which are zero from the
$n$th place onward. Clearly $R/I_n \isom R$, so the ideals are
semiprime. However $\bigcup I_n$ is composed of sequences of
matrices which are eventually zero, and $aRa$ is eventually zero
for $a = \left(\begin{pmatrix}
0 & 1  \\
0 & 0
\end{pmatrix}, \begin{pmatrix}
0 & 1  \\
0 & 0
\end{pmatrix}, \dots \right)$; hence $R/\bigcup I_n$ is not semiprime. On the other hand by the
argument in \cite{FS}, every prime ideal of $R$ is maximal, so
there are no infinite chains of primes and \PP\ holds trivially.
\end{exmpl}

\subsection{\SPSP\ does not imply \PP}

In the rest of this section we investigate \PP\ and \SPSP\ for rings of the form $\hat{A} = \smat{A}{M}{M}{A}$ where $A$ is an integral domain and
 $M \normali A$ is a nonzero ideal. We show that they always satisfy \SPSP, and give an example which does not have \PP. Clearly $\hat{A}$ is a prime ring. Let us describe the ideals of this ring.
\begin{prop} \label{whoissp}
\begin{enumerate}
\item The ideals of $\hat{A}$ have the form $\hat{I} = \smat{I_{11}}{I_{12}}{I_{21}}{I_{22}}$, where
for  $1 \le i,j\le 2$, $I_{ij} \normal A$ (not necessarily proper),
$I_{ii'} \sub M$, and $MI_{ij} \sub I_{i'j} \cap I_{ij'}$ (where $1'
= 2$ and $2' = 1$).
\item The semiprime ideals of $\hat{A}$ are of the form
\begin{equation*}
\smat{I}{M\cap I}{M \cap I'}{I'},\end{equation*} where
$I,I'$ are semiprime ideals of $A$, and $M \cap I' = M \cap I$. 
\end{enumerate}
\end{prop}
\begin{proof}
\begin{enumerate}
\item This is well known and easy.
\item Write $A_{ij} = A$ if $i = j$ and $A_{ij} = M$ otherwise. We are given an ideal $\hat{I} \normali \hat{A}$ which thus
can be written as $\hat{I} = \smat{I_{11}}{I_{12}}{I_{21}}{I_{22}}$,
satisfying the conditions of (1). Clearly $\hat{I}$ is semiprime if
and only if for every $a_{ij} \in A_{ij}$,
\begin{center}
   ( $\sum_{j,k} A_{jk} a_{ij}a_{k\ell} \sub I_{i\ell}$ for every $i,\ell$) implies ( $a_{i\ell} \in I_{i\ell}$ for every $i,\ell$).
\end{center}
Assuming that this condition holds, fix $i,j$ and choose $a_{k\ell}
= 0$ for every $(k,\ell) \neq (i,j)$; then
\begin{center}
$\COND$ for $a_{ij} \in A_{ij}$, $A_{ji} a_{ij}^2 \sub I_{ij}$ implies $a_{ij} \in I_{ij}$.
\end{center}
On the other hand if Condition $\COND$ holds and $\sum_{j,k} A_{jk}
a_{ij}a_{k\ell} \sub I_{i\ell}$ for every $i,\ell$, then in
particular $A_{ji} a_{ij}^2 \sub I_{ij}$, so each $a_{i\ell} \in
I_{i\ell}$. We conclude that $\hat{I}$ is semiprime if and only if
$\COND$ holds for every $i,j$.

We claim that $\COND$ is equivalent to $I_{ii}$ being semiprime in
$A$ with $$M \cap I_{11} \sub I_{ij},\qquad \forall i\neq j.$$
Indeed, for $i = j$, condition $\COND$ requires that the $I_{ii}$
are semiprime in~$A$. Assuming this, the condition is ``for $a_{ij} \in A_{ij}$, $Ma_{ij}^2
\sub I_{ij}$ implies $a_{ij} \in I_{ij}$ for $i \neq j$.'' In light
of the standing assumption that $a_{ij} \in A_{ij} = M$, we claim
that this is equivalent to $M \cap I_{11} \sub I_{ij}$. Indeed, for
every $b \in M$, $Mb^2 \sub I_{ij}$ iff $b \in I_{11}$ (Proof: If
$b^2 M \sub I_{ij}$ then $b^4 \in (bM)^2 = b^2M\cdot M \sub I_{ij}M \sub I_{11}$,
so $b \in I_{11}$. On the other hand if $b \in M \cap I_{11}$ then
$b^2 M \sub MI_{11} \sub I_{ij}$), so the condition becomes ``for $b \in M$, $b \in I_{11}$ implies $b \in I_{ij}$ for $i\neq j$'', as claimed.

We have shown that $\hat{I}$ is semiprime if and only if
$I_{11},I_{22}$ are semiprime in $A$ and $M \cap I_{11} \sub I_{12}
\cap I_{21}$.

Now assume that $I_{ii}$ are semiprime, and that $M \cap I_{11} \sub
I_{12} \cap I_{21}$. Denote the idealizer of an ideal $I$ by $(I:M)
= \set{x \in A \suchthat xM \sub I}$, and notice that $M \cap (I:M)
= M \cap I$ when $I$ is semiprime. But $I_{12} M \sub I_{11}$, implying
$$I_{12}\sub M \cap (I_{11}:M) = M \cap I_{11} \sub I_{12},$$ so
$I_{12} = M \cap I_{11}$ and likewise $I_{21} = M\cap I_{11}$. By
symmetry $I_{12} = M \cap I_{22}$ as well, so $M \cap I_{11} = M
\cap I_{22}$.
\end{enumerate}
\end{proof}


\forget
Recall that for ideals $I,M$ in a commutative ring $A$, $(I:M) =
\set{a \in A\suchthat aM \sub I}$.
The properties of $(I:M)$ are well known, and we review the ones
that we need.

\begin{lem}\label{compactness}
Let $A$ be a commutative ring, with an ideal $M \normali A$. For
any semiprime ideal $I \normali A$,
$$M \cap (I:M) = M \cap I.$$
\end{lem}
\begin{proof}
The inclusion $I \sub (I:M)$ is trivial. In the other direction let $x \in M \cap (I:M)$, then $x^2 \in xM \sub I$, so $x \in I$ by assumption.
\end{proof}

\begin{lem}\label{prep}
Let $A$ be a commutative ring with an ideal $M \normali A$. Let $I,J \sub M$ be ideals of $A$ such that $I$ is semiprime, $M I\sub J$ and $M J\sub I$.

\begin{enumerate}
\item\label{x1} For $b \in M$, $b^2 M \sub J$ iff $b \in I$.

\item\label{x4} The condition ``For every $b \in M$, if $b^2M \sub
J$, then $b \in I$"
is equivalent to $M \cap I \sub J$.

\item\label{x3} Let $I'$ be another semiprime ideal such that $MI' \sub J$ and $MJ \sub I'$. Then $M \cap I = M \cap I'$.

\end{enumerate}
\end{lem}
\begin{proof}
\begin{enumerate}
\item If $b^2 M \sub J$ then $(bM)^2 \sub JM \sub I$ so $bM \sub I$ and $b \in M \cap (I:M) \sub I$ by \Pref{compactness}. On the other hand if $b \in M \cap I$ then $b^2 \in MI \sub J$ and $b^2M \sub J$.
\item This is \eq{x1}.
\item Indeed, by \Pref{compactness} and \eq{x1}, $b \in M \cap I$ iff $b \in M \cap (I:M)$ iff ($b \in M$ and $b^2M \sub J$) iff $b \in M \cap (I':M)$ iff $b \in M \cap I'$.
\end{enumerate}
\end{proof}
\forgotten

\begin{prop}\label{main0}
 The ring $\hat{A}$ satisfies
\SPSP.
\end{prop}
\begin{proof}
By \Pref{whoissp} every chain of semiprime ideals $T_1 \sub T_2 \sub
\cdots$ in $\hat{A}$ has the form $T_n =
\smat{I_n}{J_n}{J_n}{I'_n}$, $I_n$ and $I'_n$ are ascending chains
of semiprime ideals of~$A$, and $J_n = M \cap I_n = M \cap I_n'$.
The union of this chain is $\smat{\bigcup I_n}{L}{L}{\bigcup I_n'}$
where $L = M \cap \bigcup I_n = M \cap \bigcup I_n'$, which is
semiprime.
\end{proof}

{}Using the description of the semiprime ideals, it is not difficult to obtain the following.
\begin{prop}\label{main0.2}
\begin{enumerate}
\item \label{Y5} The prime ideals of $\hat{A}$ are $\smat{J}{M}{M}{A}$ and $\smat{A}{M}{M}{J}$ for prime ideals $J \normali A$ containing $M$, and $I^{0} = \smat{I}{M \cap I}{M \cap I}{I}$ for prime ideals $I \normali A$ not containing $M$.
%
%
\item The \almostprime\ ideals of $\hat{A}$ are of the form
$\smat{M'}{M}{M}{M'}$ where $M' \normali A$ is a prime ideal
containing $M$, which can be written as a union of an ascending
chain of primes not containing $M$.
\end{enumerate}
\end{prop}
\begin{proof}
 Notation as in \Pref{whoissp}(1), we are done by \Pref{whoissp}(2)
 unless some ${I_{11}} = A$, in which case we must have (1).
 If the chain of primes includes an ideal with $A$ in one of the corners, then every higher term has the
  same form, and the union is determined by the union of the ideals in the other corner, which is prime since $A$ is
   commutative. We thus assume that the chain has the form $I_1^0 \sub I_2^0 \sub \cdots$ where $I_1 \sub I_2 \sub \cdots$
    are primes in $A$, not containing $M$. The union is clearly $\hat I^0$ where $\hat I = \bigcup I_n$ is prime, and  $\hat I^0$ is not a prime iff $M \sub \hat I$.
\end{proof}

Suppose $A$ is a prime PI-ring, integral over its center $C$. In \cite{BV} it is shown that $A$ satisfies the properties Lying Over and Going Up over $C$, which gives a correspondence of chains of primes between the two rings. The next example shows that the union of chains in not preserved.

\begin{exmpl}[A prime PI-ring, integral over its center, satisfying \SPSP\ but not (uniquely-)\PP]\label{Ex4}
Let $F$ be a field. Let $\hat{A} = \smat{A}{M}{M}{A}$, where $A =
F[\lam_1,\dots]$ is the ring of polynomials in countably many
variables $\lam_1,\lam_2,\dots$, and $M = \ideal{\lam_1,\dots}$.
Clearly $\hat{A} \subset M_2(A)$ is integral over $A$. Choose $I_n
= \ideal{\lam_1,\dots,\lam_n}$. Then $T_n =
\smat{I_n}{I_n}{I_n}{I_n} \normali \hat{A}$ form an ascending
chain of primes by \Pref{main0.2}.\eq{Y5}, 
but their union $\bigcup T_n = \smat{M}{M}{M}{M}$ is obviously not
prime. Therefore $\hat{A}$ satisfies \SPSP\ (\Pref{main0}) but not
\PP. Furthermore $\hat{A}/\bigcup T_n\cong A/M\times A/M$ which has
two minimal ideals, so uniquely-\PP\ also fails.
\end{exmpl}

\section{The $\PP$-index}\label{sec:PP}

Let $\tilde P = \set{P_{\alpha}}$ be a chain of prime ideals in a
ring $R$. The number of non-prime unions of subchains of $\tilde P$
is called the {\bf{index}} of $\tilde P$ (either finite or
infinite).
The {\bf{\PP-index}} of $R$, denoted by $\PP(R)$, is the supremum of the indices of all chains of primes in $R$.

\begin{prop}
For any ring $R$, $\PP(R) = \sup \PP(R/P)$ where $P\normali R$ ranges over the prime ideals.
\end{prop}
\begin{proof}
By definition $\PP(R/P)$ is the supremum of indices of chains of
primes containing $P$. Therefore, the supremum on the right-hand
side is the supremum of indices of chains of primes containing some
prime $P$, but any chain contains its own intersection, so this
supremum is by definition $\PP(R)$.
\end{proof}

\begin{rem}
If $n = \PP(R)$ then $R$ has a chain of $n$ \almostprime{s}. It is not clear if the converse holds. For example if $\set{P_\lam}$ and $\set{P_{\lam}'}$ are ascending chains of primes such that $\bigcup P_{\lam} \subset \bigcup P_{\lam}'$ are not primes, does it follow that there is a chain of primes with at least two non-prime unions?
\end{rem}

We claim that:
\begin{prop}\label{PPindex}
For any ring $R$,
$$\PPind(R) = \left\{ \begin{array}{cl} 0 & i\!f\mbox{\, $R$\, has the property\, \PP} \\ \sup_I \PPind(R/I)+1 & \mbox{otherwise} \end{array} \right.$$
where the supremum is taken over the \almostprime\ ideals of $R$
(when they exist).
\end{prop}
\begin{proof}
Indeed, $\PPind(R) = 0$ if and only if there are no \almostprime{}s
(which are non-prime by definition), if and only if $R$ satisfies
$\PP$. Now assume $R$ does not satisfy~\PP. Consider the set
$\set{\PPind(R/I)}$ ranging over the \almostprime{}s $I$.  If this set is unbounded, then clearly $\PPind(R) = \infty$. Otherwise, take a \almostprime{} $I$ such
that $n = \PPind(R/I)$ is maximal among the \PP-indices of the
quotients. If $J_1/I \subset \cdots \subset J_n/I$ are
\almostprime{}s in a chain of primes in $R/I$, then $I \subset J_1
\subset \cdots \subset J_n$ are \almostprime{s} in a chain in $R$.
On the other hand if $J_0 \subset J_1 \subset \cdots \subset J_n$
are \almostprime{s} in a chain in $R$, then $\PPind(R/J_0) \geq n$.
\end{proof}
For example, $\PPind(R) = 1$ if and only if the union of an
ascending chain of primes starting from a \almostprime\ ideal is
necessarily prime.

\section{The property $\PP$ in PI-rings}\label{sec:PI}

In this section we show that for PI-rings, the $\PP$-index is
bounded by the PI-class.

\begin{prop}\label{Azumaya_CP}
Any Azumaya algebra satisfies \PP (and \SPSP).
\end{prop}
\begin{proof}
Let $A$ be an Azumaya algebra over a commutative ring $C$. There is
a 1:1 correspondence between ideals of $A$ and the ideals of $C$,
preserving inclusion, primality and semiprimality. The claim follows
since the center satisfies \PP\ (and~\SPSP).
\end{proof}

Recall that by Posner's theorem (\cite{RowenPI}), a prime PI-ring $R$ is {\emph{representable}}, namely embeddable in a matrix algebra $\M[n](C)$ over a commutative ring $C$. The minimal such $n$ is the PI-class of $R$, denoted $\PIdeg(R)$.

Although PI-rings do not necessarily satisfy the property \PP, we
show
that the PI-class bounds the extent in which \PP\ may fail. 

We are now ready for our main positive result about PI-rings.

\begin{thm}\label{mainPI}
Let $R$ be a (prime) PI-ring. Then $\PPind(R) < \PIdeg(R)$.
\end{thm}
\begin{proof}
Let $R$ be a prime PI-ring of PI-class $n$. If the PI-class is $1$
then $R$ is commutative, and has $\PPind(R) = 0$. We continue by
induction on $n$. Let $$0 = P_0 \subset P_1 \subset \cdots$$ be an
ascending chain of primes, and assume that $\bigcup P_n$ is not a
prime ideal. Let $Q \supset \bigcup P_n$ be a prime ideal. We want
to prove that the PI-class of~$R/Q$ is smaller than that of $R$.

Assume otherwise. Let $g_n$ be a central polynomial for $n \times n$
matrices (see~\cite[p.~26]{RowenPI}). Since $\PIdeg(R/Q) = n$, there
is a value $\gamma \neq 0$ of $g_n$ in the center of $R$, which is
not in $Q$. Since the center is a domain we can consider the
localization $A[\gamma^{-1}]$ (see~\cite[Section~2.12]{R-RT}), which
is Azumaya by Rowen's version of the Artin-Procesi Theorem \cite[Theorem 1.8.48]{RowenPI}, since
$1$ is a value of $g_n$ on this algebra. But then the union of $$0
\subset P_1[\gamma^{-1}] \subset P_2[\gamma^{-1}] \subset \cdots$$
is prime by \Pref{Azumaya_CP}, so $\bigcup P_n$ is prime as well,
contrary to assumption.
\end{proof}


We now show the bound is tight. Notice that the ring constructed in \Eref{Ex4} has PI-class $2$ and is not \PP (and thus has $\PPind(R) = 1$). Let us generalize this.

\begin{exmpl}[An algebra of PI-class $n$ which has \PP-index $n-1$]\label{PI_non_CP}
Let $A_{(n)}=F\left[\lambda_i^{(j)} : 1 \leq j < n,\, i =
1,2,\dots\right]$. Let $M_n = 0$ and, for $j = n-1,n-2,\dots,1$,
take
$M_j=M_{j+1}+\left<\lambda_1^{(j)},\lambda_2^{(j)},\cdots\right>$,
so that $0 = M_n \subset M_{n-1} \subset \cdots \subset M_1 \normali A_{(n)}$. Let
$e_{ij}$ denote the matrix units of the matrix algebra over $A_{(n)}$. Let $J_{(n)} =
\sum_{i,j} e_{ij}M_{\max(i,j)-1}$, $S_{(n)} = \sum e_{ii} A_{(n)}$.
Let
$$R_{(n)} = J_{(n)}+S_{(n)} =\begin{pmatrix}
A_{(n)} & M_1 & M_2 & \cdots & M_{n-1} \\
M_1 & A_{(n)} & M_2 & \cdots & M_{n-1} \\
M_2 & M_2 & A_{(n)} & \cdots & M_{n-1} \\
\vdots & \vdots & \vdots & \ddots & \vdots \\
M_{n-1} & M_{n-1} & M_{n-1} & \cdots & A_{(n)} \\
\end{pmatrix}.$$

Clearly $S_{(n)} J_{(n)}, \, J_{(n)} S_{(n)} \sub J_{(n)}$, and
$S_{(n)}S_{(n)}= S_{(n)}$. Moreover $M_{k}M_{\ell} \sub
M_{\max\set{k,\ell}}$ for every $k,\ell$, so that $J_{(n)}J_{(n)}
\sub J_{(n)}$. It follows that $R_{(n)} = J_{(n)}+S_{(n)}$ is a
ring. The ring of central fractions of $R_{(n)}$ is the simple ring
$\M[n](\operatorname{q}(A_{(n)}))$, so $R_{(n)}$ is prime, of
PI-class $n$.

When $n = 2$ we obtain the ring of \Eref{Ex4}, so $\PPind(R_{(2)}) =
1$. For arbitrary $n$, consider the chain $I_1 \subset I_2 \subset
\cdots$ of ideals of $A$ defined by $I_i =
\ideal{\lam_1^{(n-1)},\dots,\lam_i^{(n-1)}}$; thus $\bigcup I_i =
M_{n-1}$. Let $\tilde{I}_i = \M[n](I_i)$. Each ideal $\tilde{I}_i$
is prime (again by central fractions), and their union is the set of
matrices over $M_{n-1}$. The quotient ring is therefore
$R_{(n)}/\bigcup_i \tilde{I}_i \,\isom\, R_{(n-1)} \times
A_{(n-1)}$, which is not prime, and $\PPind(R_{(n-1)}) = n-2$ by
induction. We conclude that $\PPind(R_{(n)}) = n-1$.
\end{exmpl}

\begin{rem}\label{end}
\begin{enumerate}
\item\label{AffinePIhasACCsp} Although PI-rings do not necessarily satisfy \PP
(see \Eref{Ex4}), affine PI-rings over commutative Noetherian rings
do satisfy ACC on semiprime ideals, and in particular are \SPSP\ and
\PP (by Schelter's theorem, \cite[Thm~4.4.16]{RowenPI}).
\item PI-rings of finite Gel'fand-Kirillov dimension satisfy ACC on
primes, since a prime PI-ring is Goldie, and then every prime
ideal contains a regular element which reduces the dimension.
\item On the other hand, we have examples of a (non-affine) locally nilpotent monomial algebra, which does not satisfy \PP, and of affine algebras of $\GK=2$ which do not satisfy \PP (details will appear elsewhere.) 
    In both cases the \PP-index is uncountable.
\item Graded affine algebras of quadratic growth and graded affine domains of cubic growth have finite Krull dimension and so satisfy \PP, \cite{BS,GLSZ}.
\end{enumerate}
\end{rem}


\end{document}